\documentclass[11pt,a4paper]{amsart}
\usepackage[dvips]{graphicx}
\usepackage{amsmath}
\usepackage{amsfonts}
\usepackage{amsthm}
\usepackage{amssymb}
\usepackage{mathrsfs}
\usepackage[T1]{fontenc}
\usepackage[latin1]{inputenc}
\usepackage[all]{xy}
\usepackage{a4wide}
\usepackage{psfrag}
\usepackage[bf]{caption}

\newcommand{\finw}{\mathcal{W}}

\newcommand{\Pic}{\mathrm{Pic}\,}

\newcommand{\fine}{\mathcal{E}}
\newcommand{\finevee}{\mathcal{E}^{\vee}}
\newcommand{\finf}{\mathcal{F}}
\newcommand{\finfca}{\mathcal{F}_{C,A}}

\newcommand{\finecax}{\mathcal{E}_{C,A}}
\newcommand{\finecaxvee}{\mathcal{E}_{C,A}^{\vee}}
\newcommand{\fino}{\mathcal{O}}
\newcommand{\tykkp}{\mathbb{P}}

\newcommand{\tykkc}{\mathbb{C}}

\newcommand{\gon}{\mathrm{gon}}

\newcommand{\lengdexi}{\mathrm{length}(\xi)}
\newcommand{\finefine}{\mathcal{E}_{C,A}\otimes\mathscr{E}_{C,A}^{\vee}}
\newcommand{\finefineto}{\mathcal{E}\otimes\mathcal{E}^{\vee}}

\newcommand{\finixi}{\mathcal{I}_{\xi}}

\newcommand{\Cliff}{\mathrm{Cliff}}

\newcommand{\fing}{\mathcal{G}}
\newcommand{\finp}{\mathcal{P}}

\newtheorem{Prop}{Proposition}[section]
\newtheorem{Def}[Prop]{Definition}
\newtheorem{Teorem}[Prop]{Theorem}
\newtheorem{Lemma}[Prop]{Lemma}

\newtheoremstyle{Property}{\topsep}{\topsep}%
  {}
  {}
  {\bfseries}
  {.}
  { }
  {\thmname{#1}\thmnumber{ #2}\thmnote{ #3}}

\theoremstyle{Property}

\newtheoremstyle{Eksempel}{\topsep}{\topsep}%
  {}
  {}
  {\bfseries}
  {.}
  { }
  {\thmname{#1}\thmnumber{ #2}\thmnote{ #3}}

\theoremstyle{Eksempel}
\newtheorem{Eksempel}[Prop]{Example}
\newtheoremstyle{Bemerkning}{\topsep}{\topsep}%
  {}
  {}
  {\bfseries}
  {.}
  { }
  {\thmname{#1}\thmnumber{ #2}\thmnote{ #3}}

\theoremstyle{Bemerkning}
\newtheorem{Bemerkning}[Prop]{Remark}
\date{}

\title{Pencils of small degree on curves on unnodal Enriques surfaces}

\author{Nils Henry Rasmussen and Shengtian Zhou}


\begin{document}

\begin{abstract}
We use vector-bundle techniques in order to compute $\dim W^1_d(C)$ where $C$ is general and smooth in a linear system on an unnodal Enriques surface. We furthermore find new examples of smooth curves on Enriques surfaces with an infinite number of $g^1_{\gon(C)}$'s.
\end{abstract}

\maketitle

\section{Introduction}

Let $S$ be a smooth surface over $\tykkc$, and $L$ a line-bundle on $S$. Let $W^r_d(C)$ be the Brill--Noether variety, parametrising complete $g^s_d$'s on $C$ for $s\geq r$. We will be concerned with finding the dimension of $W^1_d(C)$ for small $d$ when $S$ is an unnodal Enriques surface.

The theory on the dimension of $W^r_d(C)$ dates back to 1874, when Alexander von Brill and Max Noether made an incomplete proof stating that $\dim W^r_d(C)=\rho(g,r,d):=g-(r+1)(g-d+r)$ provided $C$ is general of genus $g$. It was first much later that strict proofs for this were presented (\cite{Laksov}, \cite{Kempf}, \cite{Griffiths}). In 1987, a new proof was constructed by Lazarsfeld (\cite{Lazarsfeld}) involving use of vector-bundle techniques for curves on K3 surfaces, exploiting the fact that for any $g\geq 2$, a K3 surface with Picard group $\mathbb{Z}C$ with $C$ a smooth genus $g$ curve can be constructed. These vector-bundle techniques, which were also developed by Tuyring (\cite{Tyurin}), were later used to study the gonality and Clifford index of any smooth curve on an arbitrary K3 surface (\cite{ciliberto}, \cite{knutsen-2003}, \cite{knutsen}, \cite{Aprodu-Farkas}). These methods have also lately been applied in the case of Enriques surfaces and rational surfaces with an anticanonical pencil (\cite{knutsen-2001}, \cite{KL09}, \cite{knutsen2009secant}, \cite{lelli2012green}).

The dimension of $W^1_d(C)$ was studied in \cite{Aprodu-Farkas} and \cite{lelli2012green} because of a result by Aprodu in 2005 (\cite{Aprodu}), stating that if $\dim W^1_d(C)=d-\gon(C)$ for $d\leq g-\gon(C)+2$, then the Green and Green--Lazarsfeld conjectures are satisfied. These conjectures state that the Clifford index and gonality can be read off minimal free resolutions of $\bigoplus_n H^0(C,\fino_C(nK_C))$ and $\bigoplus_n H^0(C,\fino_C(nA))$ for $\deg(A)\gg 0$, respectively (see \cite{Green} and \cite{G-L}).

\bigskip

In this article, we make an attempt at finding the dimension of $W^1_d(C)$ when $C$ is a smooth curve on an unnodal Enriques surface $S$. A smooth surface over $\tykkc$ is an Enriques surface if $h^1(S,\fino_S)=0$, $2K_S\sim 0$ and $K_S\nsim 0$. One defines
$$\phi(L):=\min\{L.E\,|\,E\in\Pic(S),\,E^2=0,\,E\not\equiv 0\}$$
and
$$\mu(L):=\min\{L.B-2\,|\,B\in\Pic(S)\text{ with $B$ effective, $B^2=4$, $\phi(B)=2$, and $B\not\equiv L$}\}.$$
By \cite{KL09}, the generic gonality for smooth curves in $|L|$, which we denote by $k$, is given by
\begin{equation*}
k=\min\left\{2\phi(L),\,\mu(L),\,\left\lfloor\frac{L^2}{4}\right\rfloor+2\right\}.
\end{equation*}
Furthermore, $k=\mu(L)<2\phi(L)$ precisely when:
\begin{itemize}
\item $L^2=\phi(L)^2$ with $\phi(L)\geq 2$ and even, in which case $k=\mu(L)=2\phi(L)-2$; or
\item $L^2=\phi(L)^2+\phi(L)-2$ with $\phi(L)\geq 3$, $L\not\equiv 2D$ for $D$ such that $D^2=10$, $\phi(D)=3$, in which case $k=2\phi(L)-1$ for $\phi(L)\geq 5$ and $k=2\phi(L)-2$ for $\phi(L)=3,4$.
\end{itemize}
If $(L^2,\phi(L))=(30,5),\,(22,4),\,(20,4),\,(14,3),\,(12,3)$ or $(6,2)$, then $k=\left\lfloor\frac{L^2}{4}\right\rfloor+2=\phi(L)-1$.

In all other cases, $k=2\phi(L)$.

\bigskip

Our main result is the following:

\begin{Teorem}\label{main}
Let $S$ be an unnodal Enriques surface, and let $|L|$ be an ample linear system with $L^2\geq 2$ such that $k=2\phi(L)<\mu(L)$. Then, for $k\leq d\leq g-k$ and $C$ general in $|L|$,
$$\dim W^1_{d}(C)=d-k.$$
\end{Teorem}

\begin{Bemerkning}\label{non-lg}
In the case where $L=n(E_1+E_2)$ for $\geq 3$ and $E_1.E_2=2$, we have $k=\mu(L)< 2\phi(L)$, by \cite[Corollary 1.5 (a)]{KL09}. In Example \ref{non-linear} we prove that there exists a sub-linear system $\mathfrak{d}\subseteq |L|$ of smooth curves such that for general $C\in \mathfrak{d}$, there exist infinitely many $g^1_{\gon(C)}$'s. These curves are non-exceptional and are, as far as we know, new examples of curves with an infinite number of $g^1_{\gon(C)}$'s.
\end{Bemerkning}

\bigskip

\begin{Bemerkning}
A conjecture by Martens (\cite[Statement T, page 280]{Martens}) states that if $\dim W^1_{\gon(C)}(C)=0$, then $\dim W^1_d(C)=d-\gon(C)$ for $d\leq g-\gon(C)+2$; and that if $\dim W^1_{\gon(C)}(C)=1$, then $\dim W^1_d(C)=d-\gon(C)+1$ for $d\leq g-\gon(C)+2$. We therefore expect that Theorem \ref{main} is valid for $d\leq g-k+2$, and hence that the Green and Green--Lazarsfeld conjectures are satisfied for the curves in question.
\end{Bemerkning}

\bigskip

This paper is organised as follows: In Section \ref{prels}, we introduce the basic results of Brill--Noether theory and the vector-bundles associated to the pairs $(C,A)$, where $|A|$ is a $g^1_d$ on $C$. In Section \ref{non-stable}, we prove Theorem \ref{main} in the case where the general vector-bundles are non-stable, while the stable case is covered in Section \ref{stable-section}. We close with an example of a sub-linear system of curves with an infinite number of $g^1_{\gon(C)}$'s in Section \ref{example-section}.

\bigskip

\noindent\textbf{Acknowledgments.} Thanks to Andreas Leopold Knutsen for introducing us to this subject, and for valuable comments and remarks.

\section{Preliminaries}\label{prels}
\subsection{Brill--Noether theory}

Let $C$ be a smooth curve over $\tykkc$, and let $r$ and $d$ be non-negative integers. Then there is a variety $W^r_d(C)$ that parametrises all complete $g^s_d$'s on $C$, for all $s\geq r$.

Let $|A|$ be a complete $g^r_d$ on $C$, and let $\mu_{0,A}:H^0(C,\fino_C(A))\otimes H^0(C,\fino_C(K_C-A))\rightarrow H^0(C,\fino_C(K_C))$ be the cup-product mapping. (This is known as the Petri map.) Then, from \cite[IV, Proposition 4.2]{arbarello}, we have
\begin{equation}\label{rho-equality}
\dim T_{[A]}W^r_d(C)=\rho(g,r,d)+\dim\ker(\mu_{0,A}),
\end{equation}
where $\rho(g,r,d):=g-(r+1)(g-d+1)$ is called the Brill--Noether Number, and also known as ``the expected dimension of $W^r_d(C)$''.

Furthermore, if $|A|$ is base-point free and $h^0(C,\fino_C(A))=2$, then the base-point free pencil trick (\cite[page 126]{arbarello}) gives us

\begin{equation}\label{mu-equality}
\ker\mu_{0,A}=H^0(C,\fino_C(K_C-2A)).
\end{equation}

One defines the \emph{gonality} of $C$ to be the smallest $d$ such that there exists a $g^1_d$ on $C$, and denotes it by $\gon(C)$. It is known that for any smooth curve $C$ of genus $g$,
\begin{equation}\label{gon-boundary}
\gon(C)\leq\left\lfloor\frac{g+3}{2}\right\rfloor.
\end{equation}
For the general curve of genus $g$, we have equality in (\ref{gon-boundary}). Note that for curves on Enriques surfaces, since it is known that $\phi(C)\leq \sqrt{C^2}=\sqrt{2g-2}$, the gonality is usually not maximal.

Let $W$ be a component of $W^1_d(C)$ containing $A$. Then,

\begin{equation}\label{rho-lg}
\text{if }\dim\ker\mu_{0,A}=0\text{ and }d\leq g-\gon(C)+2\text{, then }\dim W\leq d-\gon(C).
\end{equation}

Also, note that if the general $g^1_d$ in $W$ has base-points, then we can obtain these $g^1_d$'s by considering $g^1_{d-1}$'s and add base-points. It follows that

\begin{equation}\label{base-points}
\text{if the general $g^1_d$'s in $W$ have base-points, then $\dim W\leq \dim W^1_{d-1}(C)+1$.}
\end{equation}

The following definition, which was introduced in \cite{Martens-1968}, generalises the notion of gonality for a curve $C$:

\begin{Def}\label{Cliff-def}
Let $C$ be a smooth curve of genus $g\geq 4$. The Clifford index of $C$ is defined to be
$$\Cliff(C):=\min\{\deg(A)-2(h^0(C,\fino_C(A))-1)\,|\,h^0(C,\fino_C(A))\geq 2\text{ and }h^1(C,\fino_C(A))\geq 2\}.$$
If $A$ is a divisor on $C$ satisfying $h^0(C,\fino_C(A))\geq 2$ and $h^1(C,\fino_C(A))\geq 2$, then one says that $A$ contributes to the Clifford index of $C$, and $A$ is then defined to have Clifford index $\Cliff(A):=\deg(A)-2(h^0(C,\fino_C(A))-1)$.

If $C$ is hyperelliptic of genus $2$ or $3$, one defines $\Cliff(C)=0$; and if $C$ is non-hyperelliptic of genus $3$, one defines $\Cliff(C)=1$.
\end{Def}

It was proved in \cite[Theorem 2.3]{C-M} that $\Cliff(C)\in\{k-2,k-3\}$, where $k=\gon(C)$. We have $\Cliff(C)=k-2=\lfloor\frac{g-1}{2}\rfloor$ if $C$ is general in $\mathscr{M}_g$ for $g\geq 2$. If $\Cliff(C)=k-3$, then $C$ is said to be \emph{exceptional}.

\subsection{Vector-bundle techniques}

Let $S$ be an Enriques surface, and let $L$ be a line-bundle on $S$. One defines $\finw^1_d|L|:=\{(C,A)\,|\,C\in |L|_s,\,A\in W^1_d(C)\}$, and $\pi:\finw^1_d|L|\rightarrow |L|_s$ the natural projection map, where $|L|_s$ denotes the smooth curves of $|L|$. Each fibre of $\pi$ is isomorphic to $W^1_d(C)$.

Let $\finw$ be an irreducible component of $\finw^1_d|L|$ such that $\pi$ restricted to $\finw$ dominates. By (\ref{base-points}), we can assume that for general $(C,A)$ in $\finw$, $|A|$ is base-point free. It thus makes sense to study the associated \emph{Lazarsfeld--Mukai vector bundles}, $\finfca$ and $\finecax$ (see \cite{Lazarsfeld}).

Let $A\in W^1_d(C)\setminus W^2_d(C)$ be base-point free. The vector-bundle $\finfca$ is defined by
\begin{equation}\label{F-sekvens}
\xymatrix@1{
0 \ar[r] &\finfca \ar[r] &H^0(S,\fino_S(A))\otimes\fino_S \ar[r]^-{\mathrm{ev}} &\fino_S(A)\ar[r] &0.
}
\end{equation}
One denotes the dual of $\finf$ by $\finf^{\vee}=\finecax$. Dualising (\ref{F-sekvens}), one gets
\begin{equation}\label{E-sekvens}
0\rightarrow H^0(S,\fino_S(A))^{\vee}\otimes\fino_S\rightarrow\finecax\rightarrow \fino_C(K_C-A+K_S|_C)\rightarrow 0.
\end{equation}

Note that because we are assuming $d\leq g-\gon(C)$, then $h^0(C,\fino_C(K_C-A+K_S|_C))>0$, by Riemann--Roch. Hence, the vector-bundles $\finecax$ are globally generated away from a finite set of points, those points being the possible base-points of $\fino_C(K_C-A+K_S|_C)$. One has the following properties of $\finecax$:
\begin{eqnarray}
&\bullet&c_1(\finecax)=L \label{c1} \\
&\bullet&c_2(\finecax)=d \label{c2} \\
&\bullet&h^0(S,\finecaxvee)=h^1(S,\finecaxvee)=0,\,h^2(S,\finecax)=0 \label{h1-and-h2} \\
&\bullet&h^1(S,\finecax)=h^0(C,\fino_C(A+K_S|_C)) \label{h1fine}
\end{eqnarray}

Given a vector-bundle $\fine$ of rank $2$, with $c_1(\fine)=L$, $c_2(\fine)=d$, and $h^2(S,\fine)=0$, and which is finitely generated away from a finite set of points, then given a two-dimensional subspace $\Lambda$ in $H^0(S,\fine)$, the cokernel of $\Lambda\otimes\fino_S\hookrightarrow \fine$ is isomorphic to $\fino_{C_{\Lambda}}(B)$ for some $C_{\Lambda}\in|L|$, and where $B$ is a torsion-free sheaf of rank $1$ on $C_\Lambda$. If $C_\Lambda$ is smooth, then $B\cong \fino_{C_\Lambda}(K_{C_\Lambda}-A_\Lambda+K_S|_{C_\Lambda})$ for some $|A|\in W^1_d(C_\Lambda)$, giving us an exact sequence
\begin{equation}\label{lambda}
0\rightarrow \Lambda\otimes\fino_S\rightarrow\fine\rightarrow \fino_{C_{\Lambda}}(K_{C_\Lambda}-A_\Lambda+K_S|_{C_\Lambda})\rightarrow 0.
\end{equation}

An important tool for us will be the following:

\begin{Prop}\label{coboundary}
Suppose that $\finw$ is a component of $\finw^1_d|L|$ such that $\pi:\finw\rightarrow |L|$ dominates. Let $(C,A)$ be sufficiently general in $\finw$, and suppose that $|A|$ is base-point free for these $A$. Then there exists an exact sequence
$$0\rightarrow H^0(C,K_S|_C)\rightarrow H^0(C,\finecaxvee\otimes\fino_C(K_C-A))\rightarrow H^0(C,\fino_C(K_C-2A))\rightarrow 0.$$
In particular, $h^0(C,\finecaxvee\otimes\fino_C(K_C-A))=\dim\ker\mu_{0,A}$.
\end{Prop}

\begin{proof}
We follow the proof of \cite[Theorem 2]{Pareschi}. (See also \cite[Proposition 3.2]{lelli2012green}.)

Since $|A|$ is base-point free and $h^0(C,\fino_C(A))=2$, we have an exact sequence
$$0\longrightarrow \fino_C(-A)\longrightarrow H^0(C,\fino_C(A))\otimes\fino_C\overset{\mathrm{ev}}{\longrightarrow} \fino_C(A)\longrightarrow 0,$$
where $\mathrm{ev}$ is the evalutation morphism.

The diagram
$$\xymatrix{
0 \ar[r] & \finecaxvee \ar[r]\ar[d] & H^0(C,\fino_C(A))\otimes\fino_S \ar[r]\ar[d] & \fino_C(A) \ar[r]\ar[d] & 0 \\
0 \ar[r] & \fino_C(-A) \ar[r]\ar[d] & H^0(C,\fino_C(A))\otimes\fino_C \ar[r]\ar[d] & \fino_C(A) \ar[r]\ar[d] & 0 \\
  & 0 & 0  & 0 & }
$$
yields a surjection $\finecaxvee|_C\rightarrow \fino_C(-A)\rightarrow 0$, and since $\bigwedge^2\finecaxvee|_C=\fino_C(-K_C+K_S|_C)$, the kernel must be $\fino_C(A-K_C+K_S|_C)$, and we get the sequence
$$0\rightarrow \fino_C(A-K_C+K_S|_C)\rightarrow \finecaxvee|_C\rightarrow \fino_C(-A)\rightarrow 0.$$
We tensor with $\fino_C(K_C-A)$ and get
\begin{equation*}
0\rightarrow\fino_C(K_S|_C)\rightarrow\finecaxvee|_C\otimes\fino_C(K_C-A)\rightarrow\fino_C(K_C-2A)\rightarrow 0.
\end{equation*}
Taking global sections gives us
\begin{multline*}
0\rightarrow H^0(C,\fino_C(K_S|_C))\rightarrow H^0(C,\finecaxvee|_C\otimes\fino_C(K_C-A))\rightarrow H^0(C,\fino_C(K_C-2A))\\
\rightarrow H^1(C,\fino_C(K_S|_C)).
\end{multline*}

Note that from (\ref{mu-equality}) we have $H^0(C,\fino_C(K_C-2A))=\ker\mu_{0,A}$. Following an argument identical to \cite[Lemma 1]{Pareschi}, we have that the coboundary-map $H^0(C,\fino_C(K_C-2A))\rightarrow H^1(C,\fino_C(K_S|_C))$ up to constant factors is equal to the map $\mu_{1,A,S}:\ker_{0,A}\rightarrow H^1(C,\fino_C(K_S|_C))$ which is given as follows:

The map $\mu_{1,A,S}$ is the composition of the Gaussian map $\mu_{1,A}:H^0(C,\fino_C(K_C-2A))\rightarrow H^0(C,\fino_C(2K_C))$ with the transpose of the Kodaira--Spencer map $\delta_{C,S}^{\vee}:H^0(C,\fino_C(2K_C))\rightarrow (T_C|L|)^{\vee}=H^1(C,N_{C|S}^{\vee}\otimes\fino_C(K_C))=H^1(C,\fino_C(K_S|_C))$.

The lemma follows from considering a commutative diagram
$$\xymatrix{
0 \ar[r] & \fino_C(K_S|_C) \ar[r]\ar@2{-}[d] & \finecaxvee|_C\otimes\fino_C(K_C-A) \ar[r]\ar[d] & \fino_C(K_C-2A) \ar[r]\ar[d]^s & 0 \\
0 \ar[r] & \fino_C(K_S|_C) \ar[r] & \Omega^1_S\otimes\fino_C(K_C) \ar[r] & \fino_C(2K_C) \ar[r]  & 0,} $$
where $\mu_{1,A}$ is found by considering $s$ on the global sections level, and $\delta_{C,S}^{\vee}$ is the coboundary map $H^0(C,\fino_C(2K_C))\rightarrow H^1(C,\fino_C(K_S|_C))$.

In \cite[page 197]{Pareschi}, it is argued that
$$\mathrm{Im}(\mathrm{d}\pi_{C,A})\subset \mathrm{Ann}(\mathrm{Im}(\mu_{1,A,S})).$$
We also have a natural inclusion
$$\mathrm{Ann}(\mathrm{Im}(\mu_{1,A,S}))\subset H^1(C,\fino_C(K_S|_C))^{\vee},$$
and the latter has dimension $g-1$.

Since by assumption $\pi$ dominates $|L|$, then by Sard's lemma, $\mathrm{d}\pi_{C,A}$ is surjective for general $(C,A)$, and so $\mathrm{Im}(\mathrm{d}\pi_{C,A})$ also has dimension $g-1$.

It follows that $\mathrm{Ann}(\mathrm{Im}(\mu_{1,A,S}))= H^1(C,\fino_C(K_S|_C))^{\vee}$, and so $\mathrm{Im}(\mu_{1,A,S})=0$. Hence, the sequence
$$0\rightarrow H^0(C,\fino_C(K_S|_C))\rightarrow H^0(C,\finf_{C,A}|_C\otimes\fino_C(K_C-A))\rightarrow H^0(C,\fino_C(K_C-2A))\rightarrow 0$$
is exact.
\end{proof}

We will prove the main theorem by considering the case where the general $\finecax$'s are $\mu_L$-stable and non-$\mu_L$-stable.
\begin{Def}
Given a line-bundle $L$ on a surface $S$, a vector-bundle $\fine$ is said to be $\mu_L$-\emph{stable} if for any sub-vector bundle $\fine'$ of rank $0<\mathrm{rk}(\fine')<\mathrm{rk}(\fine)$, we have
$$\frac{c_1(\fine').L}{\mathrm{rk}(\fine')}<\frac{c_1(\fine).L}{\mathrm{rk}(\fine)}.$$
A vector-bundle $\fine$ is said to be \emph{non-}$\mu_L$\emph{-stable} if there exists a sub-vector bundle $\fine'$ of rank $0<\mathrm{rk}(\fine')<\mathrm{rk}(\fine)$ satisfying
$$\frac{c_1(\fine').L}{\mathrm{rk}(\fine')}\geq\frac{c_1(\fine).L}{\mathrm{rk}(\fine)}.$$
\end{Def}

\subsection{Assumptions}

Throughout the article, we will be using the following assumptions:

\begin{eqnarray}
&\bullet&\dim W^1_{d-1}(C)=d-1-k \text{ for $C$ general in $|L|$ (by induction).} \label{assumption1} \\
&\bullet&\text{The general $g^1_d$'s are base-point free for general $(C,A)$ in a component}\nonumber \\
&&\text{$\finw$ of $\finw^1_d|L|$.} \label{assumption2} \\
&\bullet&\text{$k\geq 3$ (since linear growth is always satisfied for hyperelliptic curves. This}\nonumber \\
&&\text{implies that $L^2\geq 4$)}\label{nonhyperelliptic}\\
&\bullet&k\leq d\leq g-k\text{. In particular, }k\leq \frac{g}{2}\text{, since Theorem \ref{main} is otherwise trivially} \nonumber \\
&&\text{satisfied}\label{k-leq-g/2}\\
&\bullet&\text{$\finw$ is a component of $\finw^1_d|L|$ such that $\pi:\finw\rightarrow |L|$ dominates}\nonumber\\
&&\text{and for general $C\in |L|$ the fibre over $C$ has dimension }\dim W^1_d(C)\label{dominates}
\end{eqnarray}

\section{The case where the $\finecax$'s are non-$\mu_{L}$-stable}\label{non-stable}

In this section, we will assume that for general $(C,A)\in\finw$, the vector-bundles $\finecax$ are non-$\mu_L$-stable. The main result of this section is Proposition \ref{non-stable-result}, where we do a parameter count of all possible non-$\mu_L$-stable vector-bundles that satisfy the properties of $\finecax$.

We start by recalling two results, one from \cite{knutsen2007sharp} and one from \cite{KL09}, which we will be using several times throughout this section:

\begin{Teorem}[{\cite[Theorem]{knutsen2007sharp}, case of Enriques surfaces}]\label{knutsen2007sharp}
Let $S$ be an Enriques surface, and $\fino_S(D)$ a line-bundle on $S$ such that $D>0$ and $D^2\geq 0$. Then $H^1(S,\fino_S(D))\neq 0$ if and only if one of the three following occurs:
\begin{itemize}
\item[(i)] $D\sim nE$ for $E>0$ nef and primitive with $E^2=0$, $n\geq 2$ and $h^1(S,\fino_S(D))=\left\lfloor \frac{n}{2}\right\rfloor$;
\item[(ii)] $D\sim nE+K_S$ for $E>0$ nef and primitive with $E^2=0$, $n\geq 3$ and $h^1(S,\fino_S(D))=\left\lfloor \frac{n-1}{2}\right\rfloor$;
\item[(iii)] there is a divisor $\Delta>0$ such that $\Delta^2=-2$ and $\Delta.D\leq -2$.
\end{itemize}
\end{Teorem}

Note that since the Enriques surfaces in question in our article are assumed to be unnodal, then part (iii) of Theorem \ref{knutsen2007sharp} cannot occur.

\begin{Lemma}[{\cite[Lemma 2.12]{KL09}}]\label{lemma212}
Let $L>0$ be a line bundle on an Enriques surface S with $L^2\geq 0$. Then there is an integer $n$ such that $1\leq n\leq 10$ and, for any $i=1,\dots,n$, there are primitive divisors $E_i>0$ with $E_i^2=0$ and integers $a_i>0$ such that
$$L\equiv a_1E_1+\cdots+a_nE_n$$
and one of the following intesection sets occurs:
\begin{itemize}
\item[(i)] $E_i.E_j=1$ for $1\leq i<j\leq n$.
\item[(ii)] $n\geq 2$, $E_1.E_2=2$ and $E_i.E_j=1$ for $2\leq i<j\leq n$ and for $i=1$, $3\leq j\leq n$.
\item[(iii)] $n\geq 3$, $E_1.E_2=E_1.E_3=2$ and $E_i.E_j=1$ for $3\leq i<j\leq n$, for $i=1$, $4\leq j\leq n$ and for $i=2$, $3\leq j\leq n$.
\end{itemize}
\end{Lemma}

The following proposition is crucial to our result. The fact that we can assume that the vector-bundles are contained in a short-exact sequence as in (\ref{MN-sequence}), where $M.L\geq N.L$, will eventually ensure that the dimensions of extensions of various $\fino_S(M)$ and $\fino_S(N)\otimes\finixi$ is small enough to give us the desired result (see Lemma \ref{extensions}).

\begin{Prop}\label{MN-prop}
Suppose $\finecax$ is non-$\mu_L$-stable. Then there exist line-bundles $\fino_S(M)$ and $\fino_S(N)$, and a $0$-dimensional subscheme $\xi$, such that  $\finecax$ sits inside an exact sequence
\begin{equation}\label{MN-sequence}
0\rightarrow \fino_S(M)\rightarrow\finecax\rightarrow\fino_S(N)\otimes\finixi\rightarrow 0,
\end{equation}
satisfying the following conditions:
\begin{itemize}
\item[(a)] We have $M+N\sim C$, $\lengdexi=d-M.N$, and $|N|$ is non-trivial and base-component free (implying that $h^0(S,\fino_S(N))\geq 2$). Furthermore, $h^2(S,\fino_S(M-N))=0$ and $h^0(S,\fino_S(N-M))=0$, unless $M\sim N+K_S$ or $M\sim N$, respectively.
\item[(b)] We have $h^2(S,\fino_S(M))=0$ and $h^0(S,\fino_S(M))\geq 2$.
\item[(c)] We have $h^1(S,\fino_S(M))=0$.
\item[(d)] We have $N|_C\geq A$.
\item[(e)] If $\xi\neq\emptyset$, then $h^1(S,\fino_S(N))=0$ and $N^2>0$.
\end{itemize}
\end{Prop}

Note that the points where $\finecax$ is not globally generated lie along the curve $C$.

\begin{proof}
Since $\finecax$ by assumption is non-$\mu_{L}$-stable, there exists a line-bundle $\fino_S(M)$ of slope $\geq g-1$ on $C$ that injects into $\finecax$. We can assume that the injection is saturated, and so we obtain the sequence (\ref{MN-sequence}). Note that since $M.C\geq g-1$, then $N.C\leq g-1$.

We have $M+N\sim C$, $\lengdexi=d-M.N$ and $|N|$ base-point free because of (\ref{c1}) and (\ref{c2}), and the fact that $\finecax$ is gobally generated away from a finite set of points. We have $N$ non-trivial because otherwise, $M.N=0$, implying that $\xi\neq\emptyset$, and this would imply that $h^0(S,\fino_S(N)\otimes\finixi)=0$, which contradicts $\finecax$ being globally generated away from a finite set.

We have $h^2(S,\fino_S(M-N))=0$ and $h^0(S,\fino_S(N-M))=0$ by the Nakai--Moishezon criterion, using that $L$ by assumption is ample, and that $M.L\geq N.L$.





We now show part (b). Since $M.C\geq g-1$, then since $L$ is ample, $-M+K_S$ cannot be effective. It follows that $h^0(S,\fino_S(-M+K_S))=0$, and this equals $h^2(S,\fino_S(M))$ by Serre duality.

To prove that $h^0(S,\fino_S(M))\geq 2$, by part (a), we have $M.N\leq d\leq g-k<g-1$, and this gives us $g-1\leq M.C=M^2+M.N<M^2+g-1$, yielding $M^2>0$. The result now follows from Riemann--Roch.


Part (c) follows from the fact that $M^2>0$ (proven in part (b)) together with Theorem \ref{knutsen2007sharp}.

To prove (d), note that by tensoring (\ref{MN-sequence}) with $\fino_S(-M)$ and taking global sections, we get $h^0(S,\finecax\otimes\fino_S(-M))\geq 1$. Rewrite (\ref{E-sekvens}) as
$$0\rightarrow\fino_S^{\oplus 2}\rightarrow\finecax\rightarrow \fino_C(C|_C-A)\rightarrow 0,$$
tensor with $\fino_S(-M)$ and take global sections. This gives us an injection $H^0(S,\finecax\otimes\fino_S(-M))\hookrightarrow H^0(C,\fino_C(N|_C-A))$, proving that $N|_C-A\geq 0$.

As for (e), suppose that $h^1(S,\fino_S(N))>0$. By Theorem \ref{knutsen2007sharp}, it follows that $N^2=0$. From (d), we then have $d\leq N.C=N.(M+N)=M.N$, contradicting part (a), which states that $d=M.N+\lengdexi$.
\end{proof}

The following lemma gives us an upper bound on $h^0(S,\fine)$, because of (\ref{MN-sequence}).

\begin{Lemma}\label{h1N}
Suppose that for general $(C,A)\in \finw$, the associated vector-bundle $\finecax$ is non-$\mu_L$-stable, so that we have a short-exact sequence as in Proposition \ref{MN-prop} where $M$ and $N$ are fixed. Then, for general $(C,A)$, we have $h^1(S,\finecax)\leq 2$ and $h^1(S,\fino_S(N)\otimes\finixi)\leq 2$.
\end{Lemma}

\begin{proof}
Note that from (\ref{dominates}), we are assuming that $\pi:\finw\rightarrow |L|$ dominates, and that for general $C\in |L|$, the fibre over $C$ has dimension $W^1_d(C)$.

Suppose $h^1(S,\fino_S(N)\otimes\finixi)\geq 3$. Then taking cohomology of (\ref{MN-sequence}), we get a surjection $H^1(S,\finecax)\rightarrow H^1(S,\fino_S(N)\otimes\finixi)\rightarrow 0$, implying that $h^1(S,\finecax)\geq 3$.

By (\ref{h1fine}), $h^1(S,\finecax)=h^0(C,\fino_C(A+K_S|_C))$, giving us $W^1_d(C)$ dimensions of $g^2_d$'s, which is impossible.
\end{proof}

The following lemma is necessary for the proof of Proposition \ref{MN-bound}, where we prove that $M.N\geq k-1$. This lemma is (in the Enriques surface case) an improvement of a similar result in \cite{Aprodu-Farkas}, where it is shown that $M|_C$ contributes to the Clifford index. By using $M|_C$ instead of $(M+E)|_C$ in Proposition \ref{MN-bound}, we would only obtain $M.N\geq k-2$.

\begin{Lemma}\label{Clifford-index}
Suppose we have a sequence as in Proposition \ref{MN-prop} with $(M-N)^2\geq 0$. If there exists a primitive elliptic curve $E$ such that $(M-N).E>0$ and $h^0(S,\fino_S(N-E+K_S))\geq 2$, then $(M+E)|_C$ contributes to $\Cliff(C)$.
\end{Lemma}

\begin{proof}
By (\ref{nonhyperelliptic}), we have $k\geq 3$, and so there exist line-bundles on $C$ that contribute to $\Cliff(C)$.

We show that $h^i(C,\fino_S(M+E)|_C)\geq 2$ for $i=0,1$.

Consider the exact sequence
\begin{equation}\label{C-sequence}
0\rightarrow\fino_S(-C)\rightarrow \fino_S\rightarrow\fino_C\rightarrow 0
\end{equation}
tensored with $\fino_S(M+E)$, giving us
$$0\rightarrow \fino_S(-N+E)\rightarrow \fino_S(M+E)\rightarrow \fino_S(M+E)|_C\rightarrow 0.$$
Because $h^0(S,\fino_S(N))\geq 2$ by Proposition \ref{MN-prop}, we must have $h^0(S,\fino_S(-N+E))=0$. By the same proposition, it follows that $h^0(S,\fino_S(M+E))\geq 2$, and so also $h^0(C,\fino_S(M+E)|_C)\geq 2$, as desired.

We have $h^1(C,\fino_S(M+E)|_C)=h^0(C,\fino_C(K_C-M|_C-E|_C))=h^0(C,\fino_S(C+K_S-M-E)|_C)=h^0(C,\fino_S(N-E+K_S)|_C)$. By considering the sequence (\ref{C-sequence}) tensored with $\fino_S(N-E+K_S)$, we get
$$0\rightarrow \fino_S(-M-E+K_S)\rightarrow \fino_S(N-E+K_S)\rightarrow \fino_S(N-E+K_S)|_C\rightarrow 0.$$
Since $h^0(S,\fino_S(M))\geq 2$ by Proposition \ref{MN-prop}, $h^0(S,\fino_S(-M-E+K_S))=0$, and so $H^0(S,\fino_S(N-E+K_S))\hookrightarrow H^0(C,\fino_S(N-E+K_S)|_C)$. We have $h^0(S,\fino_S(N-E+K_S))\geq 2$ by assumption, and so $h^0(C,\fino_S(N-E+K_S)|_C)\geq 2$ as well.
\end{proof}

In the following proposition, we obtain a connection between $M.N$ and the generic gonality in $|L|$. This is used when we make the parameter count of extensions of $\fino_S(M)$ and $\fino_S(N)\otimes\finixi$ in the proof of Proposition \ref{non-stable-result}.

Here we use that the general curves in $|L|$ are non-exceptional. This is a consequence of \cite[Corollary 1.2]{KL13}, where we find that the Clifford index for the general curve in $|L|$ is $k-2$.

Note that Proposition \ref{MN-bound}, in the case of $N^2=4$, is the only place where we use the assumption that $\mu(L)>2\phi(L)$.

\begin{Prop}\label{MN-bound}
Suppose that for general $(C,A)$ in $\finw$, the vector-bundle $\finecax$ is non-$\mu_L$-stable. Suppose furthermore that $\mu(L)>2\phi(L)=k$. Then $M.N\geq k-1$.
\end{Prop}

\begin{proof}
Suppose first that $(M-N)^2\leq 0$. Then $2M.N\geq M^2+N^2$, and so $2g-2=(M+N)^2=M^2+2M.N+N^2\leq 4M.N$, yielding $M.N\geq \frac{g-1}{2}$. Since $k\leq\frac{g}{2}$ by assumption (\ref{k-leq-g/2}), the result follows.

Now suppose $(M-N)^2> 0$.

We start by considering three special cases, namely $N^2=0$, $N^2=2$, and $N^2=4$ with $\phi(N)=2$.

\textbf{Special case 1.} If $N^2=0$, then we know from Proposition \ref{MN-prop} that $d\leq N.C=N.(M+N)=M.N$, and so it follows in particular that $k-1\leq M.N$.

\textbf{Special case 2.} If $N^2=2$, note that by Theorem \ref{knutsen2007sharp}, $h^1(S,\fino_S(N))=0$, so that $h^0(S,\fino_S(N))=\frac{1}{2}N^2+1=2$. Since $\finecax$ is globally generated away from a finite set of points, then so must $|\fino_S(N)\otimes\finixi|$, and so all points of $\xi$ must be along base-points of $|N|$ (or else $\dim|\fino_S(N)\otimes\finixi|\leq 0$). Since $\finecax$ is globally generated outside of $C$, this implies that the base-points of $|N|$, and hence also the points of $\xi$, must lie along $C$. However, $h^1(S,\fino_S(N)\otimes\finixi)$ indicates (in this particular case) how many points of $\xi$ that lie along base-points of $|N|$, and by Lemma \ref{h1N}, it follows that $\lengdexi\leq 2$.

If $d>k$, then this yields $2\geq \lengdexi=d-M.N>k-M.N$, which leads to $M.N\geq k-1$.

If $d=k$, then note that since $h^0(C,\fino_S(N)|_C)\geq 2$, then $N.C\geq k+\lengdexi$, or else we get a contradiction on the gonality of $C$. But this gives us $N.C=N^2+M.N=2+M.N\geq k+k-M.N$, yielding $M.N\geq k-1$.

\textbf{Special case 3.} In the case where $N^2=4$ with $\phi(N)=2$, we have by assumption that $N.C-2\geq \mu(L)\geq k+1$. Since $N.C=M.N+N^2=N.M+4$, it follows that $M.N\geq k-1$.

\textbf{General case.} Now suppose $M\gneq N$, $N^2\geq 4$ and $(M-N)^2>0$. If $N^2=4$, we suppose that $\phi(N)\neq 2$. We first show that there exists an elliptic curve $E$ such that the conditions of Lemma \ref{Clifford-index} are satisfied.

By Lemma \ref{lemma212}, $M-N\equiv a_1E_1+\cdots +a_mE_m$ for some elliptic curves $E_i$ satisfying $E_i.E_j>0$ for $i\neq j$, and integers $a_i>0$. Since $(M-N)^2>0$, we must have $m\geq 2$, and so $(M-N).E>0$ for any elliptic curve $E$.

Applying Lemma \ref{lemma212} again, we see that there exist positive integers $b_i$ and elliptic curves $E_i'$ satisfying $1\leq E_i'.E_j'\leq 2$ for $i\neq j$, such that $N\equiv b_1E_1'+\cdots +b_{m'}E_{m'}'$. This implies that $N-(b_1E_1'+\cdots +b_{m'}E_{m'}')$ is linearly equivalent to either $0$ or $K_S$. If $h^0(S,\fino_S(N-E))= 1$, where $E<N$, the only way this can happen is that $N-E$ is linearly equivalent to a stationary elliptic curve or a sum of two elliptic curves $E_1''$, $E_2''$ satisfying $2E_1''\sim 2E_2''$. But the first case implies $N^2\leq 4$ with $\phi(N)=2$, which is a contradiction; and in the second case, $N\sim E+E_1''+E_2''$ such that $h^0(S,\fino_S(N-E_i''))\geq 2$ for $i=1,2$, so that Lemma \ref{Clifford-index} can still be applied.

It follows that there exists an elliptic curve $E$ such that $(M+E)|_C$ contributes to $\Cliff(C)$. By \cite[Corollary 1.2]{KL13}, the general curve $C$ in $|L|$ has Clifford index $k-2$. Recalling from the proof of Lemma \ref{Clifford-index} that $H^0(S,\fino_S(M+E))\hookrightarrow H^0(C,\fino_C(M+E))$, we get
\begin{eqnarray*}
k-2&=&\Cliff(C)\\
&\leq&\Cliff((M+E)|_C)=(M+E).C-2(h^0(C,\fino_C(M+E))-1)\\
&\leq&M.(M+N)+E.C-(M+E)^2\\
&=&M.N+E.C-2M.E\\
&=&M.N+E.M+E.N-2M.E\\
&=&M.N-E.(M-N)\\
&\leq&M.N-1,
\end{eqnarray*}
as desired.
\end{proof}

The two following lemmas are used in the proof of Proposition \ref{non-stable-result}. The first one gives a parameter space for the extensions of $\fino_S(M)$ and $\fino_S(N)\otimes\finixi$, and is the second place in this article where we use the assumption that $L$ is ample. The second lemma is important when we count the dimensions of possible pairs $(C,A)$ that can arise from the same vector-bundle $\fine$.

\begin{Lemma}\label{extensions}
Suppose $\fino_S(M)$ and $\fino_S(N)$ are two line-bundles on $S$ satisfying $M.L\geq N.L$, and let $\xi$ be a non-empty zero-dimensional subscheme on $S$ of length $\ell> 0$. Then all isomorphism-classes of extensions of $\fino_S(M)$ and $\fino_S(N)\otimes\finixi$ are parametrised by
$$\tykkp Ext^1(\fino_S(N)\otimes\finixi,\fino_S(M))\cong \tykkp H^1(S,\fino_S(N-M+K_S)\otimes\finixi)^{\vee},$$
which has dimension
$$\ell+h^1(S,\fino_S(M-N))-h^2(S,\fino_S(M-N))-1.$$
\end{Lemma}

\begin{proof}
The isomorphism classes of extensions of $\fino_S(M)$ and $\fino_S(N)\otimes\finixi$ are parametrised by $\tykkp Ext^1(\fino_S(N)\otimes\finixi,\fino_S(M))\cong \tykkp H^1(S,\fino_S(N-M+K_S)\otimes\finixi)^{\vee}$, by \cite[pages 36 and 39]{Friedman}.

To find an expression for $h^1(S,\fino_S(N-M+K_S)\otimes\finixi)$, we tensor the exact sequence
$$0\rightarrow \finixi\rightarrow\fino_S\rightarrow\fino_{\xi}\rightarrow 0$$
with $\fino_S(N-M+K_S)$ and take global sections, yielding
\begin{multline*}
0\rightarrow H^0(S,\fino_S(N-M+K_S)\otimes\finixi)\rightarrow H^0(S,\fino_S(N-M+K_S))\rightarrow\mathbb{C}^{\ell} \\
\rightarrow H^1(S,\fino_S(N-M+K_S)\otimes\finixi)\rightarrow H^1(S,\fino_S(N-M+K_S))\rightarrow 0.
\end{multline*}
We have $h^0(S,\fino_S(N-M+K_S)\otimes\finixi)=0$ because $(N-M+K_S).L\leq 0$ by assumption, and using that $L$ is ample together with the Nakai--Moishezon criterion.

The result now follows by Serre duality.
\end{proof}

\begin{Lemma}\label{M-N}
Suppose $\fine$ is an extension of $\fino_S(M)$ and $\fino_S(N)\otimes\finixi$ so that we have a sequence as in Proposition \ref{MN-prop}. Then $h^0(S,\finefineto)\geq h^0(S,\fino_S(M-N))$.
\end{Lemma}

\begin{proof}
If $M\not> N$ or $M\sim N+K_S$, we have $h^0(S,\fino_S(M-N))\leq 1$, so there is nothing to prove. So assume that $M\gneq N$.

Tensor the sequence
\begin{equation}\label{MN-sequence-2}
0\rightarrow \fino_S(M)\rightarrow\fine\rightarrow\fino_S(N)\otimes\finixi\rightarrow 0
\end{equation}
by $\finevee$. Taking global sections, we see that $h^0(S,\finefineto)\geq h^0(S,\finevee\otimes\fino_S(M))$. By Serre duality, we have $h^0(S,\finevee\otimes\fino_S(M))=h^2(S,\fine\otimes\fino_S(-M+K_S))$. It thus suffices to prove that $h^2(S,\fine\otimes\fino_S(-M+K_S))\geq h^0(S,\fino_S(M-N))$.

Tensor (\ref{MN-sequence-2}) with $\fino_S(-M+K_S)$. Taking cohomology, we get
$$H^2(S,\fine\otimes\fino_S(-M+K_S))\rightarrow H^2(S,\fino_S(N-M+K_S)\otimes\finixi)\rightarrow 0.$$
So we have $h^2(S,\fine\otimes\fino_S(-M+K_S))\geq h^2(S,\fino_S(N-M+K_S)\otimes\finixi)$.

But if we consider
$$0\rightarrow\finixi\rightarrow\fino_S\rightarrow\fino_{\xi}\rightarrow 0$$
tensored with $\fino_S(N-M+K_S)$ and take cohomology, we see that $h^2(S,\fino_S(N-M+K_S)\otimes\finixi)=h^2(S,\fino_S(N-M+K_S))$, which by Serre duality equals $h^0(S,\fino_S(M-N))$. The result follows.
\end{proof}

We are now ready to state and prove the main result of this section.

\begin{Prop}\label{non-stable-result}
Suppose that for general $(C,A)$ in $\finw$, the vector-bundles $\finecax$ are non-$\mu_L$-stable, and suppose that $\mu(L)>2\phi(L)=k$. Then $\dim\finw\leq g-1+d-k$.
\end{Prop}

\begin{proof}
By assumption, for general $(C,A)$ in $\finw$, $\finecax$ sits inside an exact sequence as in Proposition (\ref{MN-prop}). We prove the proposition by making a parameter-count of all pairs $(C,A)$ such that $\finecax$ is non-$\mu_L$-stable, making a similar construction as the one done in \cite[Section 3]{Aprodu-Farkas} in the case of non-simple vector-bundles on K3-surfaces.

We divide this proof into three cases. We first consider the case where the vector-bundles $\finecax$ are indecomposable with $\ell>0$, followed by the indecomposable case when $\ell=0$. Finally, we consider the case where the $\finecax$'s are decomposable.

\bigskip

\textbf{The case where the general $\finecax$'s are indecomposable with $\ell>0$.} Fix a line-bundle $\fino_S(N)$ such that $|N|$ is base-component free, and which satisfies the following conditions: $(L-N).L\geq N.L$, $d\geq (L-N).N$, and $d-(L-N).N\leq h^0(S,\fino_S(N))$. Set $M:=L-N$ and $\ell:=d-M.N$. Note that these conditions imply that $h^1(S,\fino_S(M))=0$, $h^0(S,\fino_S(M))\geq 2$ and $h^2(S,\fino_S(M))=0$.

Let $\tilde{\finp}_{N,\ell}$ be the family of vector-bundles that are extensions of $\fino_S(M)$ and $\fino_S(N)\otimes\finixi$ where $\xi$ is a zero-dimensional subscheme of length $\ell$. For $0\leq i\leq 2$ (see Lemma \ref{h1N}), define
\begin{multline*}
\finp_{N,\ell,i}:=\{[\fine]\in \tilde{\finp}_{N,\ell}\,|\,h^2(S,\fine)=0,\,h^1(S,\fine)=i, \\
\text{ and $\fine$ is globally generated away from a finite set of points}\}.
\end{multline*}
We can think of $\finp_{N,\ell,i}$ as extensions of $\fino_S(M)$ and $\fino_S(N)\otimes\finixi$ where $\xi$ imposes $\ell-i$ conditions on $|N|$. Note that this puts restrictions on the dimensions of possible $\xi$'s that can be considered. Whereas the Hilbert scheme $S^{[\ell]}$ parametrises all possible $\xi$'s of length $\ell$, the $\xi$'s that impose $\ell-i$ conditions on $|N|$ can be found by considering elements $\eta$ of $S^{[\ell-i]}$ and add base-points of $|\fino_S(N)\otimes \mathcal{I}_{\eta}|$. Since $\fine$ is globally generated away from a finite set of points, then $|\fino_S(N)\otimes\finixi|$ is base-component free, and so there are only a finite set of base-points in $|\fino_S(N)\otimes\mathcal{I}_{\eta}|$.

\begin{itemize}
\item[($\ast$)]It follows that there are at most $2\ell-2i$ dimensions of $\xi$'s in $S^{[\ell]}$ that
impose $\ell-i$ conditions on $|N|$.
\end{itemize}

Still following the construction of \cite[Section3]{Aprodu-Farkas}, we let $\fing_{N,\ell,i}$ be the Grassmann bundle over $\finp_{N,\ell,i}$ classifying pairs $([\fine],\Lambda)$ with $[\fine]\in\finp_{N,\ell,i}$ and $\Lambda\in\mathrm{Gr}(2,h^0(S,\fine))$. (Note that $h^0(S,\fine)=h^0(S,\fino_S(M))+h^0(S,\fino_S(N))-\ell+i$, and is thus constant.)

By assumption, we have a rational map
$$h_{N,\ell,i}:\fing_{N,\ell,i}\dashrightarrow \finw^1_d(|L|)$$
given by $h_{N,\ell,i}([\fine],\Lambda):=(C_{\Lambda},A_{\Lambda})$ (see sequence (\ref{lambda})). The dimension of each fibre of $h_{N,\ell,i}$ is found by finding the dimension of all surjections $\fine\rightarrow \fino_C(K_{C_{\Lambda}}-A_\Lambda+K_S|_{C_\Lambda})$ and subtract the dimension of all morphisms from $\fino_C(K_{D_{\Lambda}}-A_\Lambda+K_S|_{C_\Lambda})$ to itself (which is $1$).

By tensoring (\ref{lambda}) with $\finevee$ and taking global sections, we see that $h^0(C,\finevee\otimes\fino_C(K_{C_{\Lambda}}-A_\Lambda+K_S|_{C_\Lambda}))=h^0(S,\finefineto)$. Since the general morphism from $\fine$ to $\fino_C(K_{C_{\Lambda}}-A_\Lambda+K_S|_{C_\Lambda})$ is surjective, it follows that the dimension of each fibre of $h_{N,\ell,i}$ is equal to $h^0(S,\finefineto)-1$. By Lemma \ref{M-N}, this is $\geq h^0(S,\fino_S(M-N))-1$.

Letting $e$ be $h^0(S,\fine)$ for any vector-bundle $\fine$ in $\finp_{N,\ell,i}$, we conclude that $\dim\finw^1_d|L|$ is bounded by $\dim\finp_{N,\ell,i}+\dim\mathrm{Gr}(2,e)-h^0(S,\fino_S(M-N))+1$.

By ($\ast$) combined with Lemma \ref{extensions}, and using that $\ell=d-M.N$, we have $\dim\finp_{N,\ell,i}\leq 2\ell-2i+\ell+h^1(S,\fino_S(M-N))-h^2(S,\fino_S(M-N))-1=3d-3M.N-2i+h^1(S,\fino_S(M-N))-h^2(S,\fino_S(M-N))-1$. We furthermore have $\dim\mathrm{Gr}(2,e)=2(e-2)=2e-4=2\chi(S,\fine)+2i-4=2(g+1-d)+2i-4=2g-2d+2i-2$.

This gives us in total
\begin{multline*}
\dim\finw\leq 3d-3M.N-2i-\chi(S,\fino_S(M-N))-1+2g-2d+2i-1\\
=2g-3M.N+d-2-\chi(S,\fino_S(M+N))+2M.N\\=2g-M.N+d-2-g
=g-2+d-M.N.
\end{multline*}
By Proposition \ref{MN-bound}, $M.N\geq k-1$, and it follows that
$$\dim\finw\leq g-1+d-k.$$

\bigskip







\textbf{The case where the general $\finecax$'s are indecomposable with $\ell=0$.} In this case, we also construct the same family $\finp_{N,\ell,i}=\finp_{N,0,i}$ of vector-bundles as in the previous case.  By \cite{Friedman}, all extensions of $\fino_S(M)$ and $\fino_S(N)$ are parametrised by $\tykkp H^1(S,\fino_S(N-M+K_S))^{\vee}=\tykkp H^1(S,\fino_S(M-N))$. As in the previous case, we consider the same family of vector-bundles $\finp_{N,0,i}$, where $i\leq 2$, together with the grassmannian bundle $\fing_{N,0,i}$. This gives us the bound
$$\dim \finw^1_d|L|\leq h^1(S,\fino_S(M-N))-1+\dim\mathrm{Gr}(2,e)-h^0(S,\fino_S(M-N))+1,$$
where $e=h^0(S,\fine)$ for the extensions $\fine$ with $h^1(S,\fine)=i$.

Since there are no indecomposable extensions of $\fino_S(M)$ and $\fino_S(N)$ when $M\sim N+K_S$ (since then $h^1(S,\fino_S(M-N))=0$), we can by Proposition \ref{MN-prop} (a) assume that $h^2(S,\fino_S(M-N))=0$. We have, as before, $\dim\mathrm{Gr}(2,e)=2g-2d+2i-2$, and $\chi(\fino_S(M-N))=g-2M.N$.

Note that since $\ell=d-M.N$, we have $d=M.N$ in this case. It follows that
$$\dim\finw^1_d|L|\leq -\chi(\fino_S(M-N))+2g-2d+2i-2+1=-g+2M.N+2g-2d+2i-1=g-1+2i.$$

Now, if $i=0$, we are done. So suppose $i>0$. Since $h^1(S,\fino_S(M))=0$ (by Proposition \ref{MN-prop}), it follows that $h^1(S,\fino_S(N))=i$, and so by Theorem \ref{knutsen2007sharp}, $N^2=0$. But then, $N.C=N.M=d$, and since $N|_C\geq A$ (by Proposition \ref{MN-prop}), it follows that these vector-bundles only yield one single $g^1_d$ for each curve $C$.

\bigskip





\textbf{The case where the general $\finecax$'s are decomposable.} Now suppose $\finecax$ is decomposable for general $(A,C)$. In that case, we must have $\ell=0$, and so $M.N=d$. Note also that there can only be finitely many different $\finecax$ in this case, and so we will here show that the image of the map $f_{\fine}:\mathrm{Gr}(h^0(S,\fine),2)\dashrightarrow \finw^1_d|L|$ is of dimension at most $g-1$, thus implying that $\dim W^1_d(C)=0$ (given the assumptions in the proposition).

As argued in the indecomposable case, we have
$$\dim\mathrm{im}\,f_{\fine}=2(h^0(S,\fine)-2)-h^0(S,\finefineto)+1.$$
Since $\fine$ is decomposable, we have $\finefineto\cong\fino_S^{\oplus 2}\oplus \fino_S(M-N)\oplus\fino_S(N-M)$.

By Proposition \ref{MN-prop}, we have either $M\sim N+K_S$ or $h^2(S,\fino_S(M-N))=0$.

If $M\sim N+K_S$, then both $M^2>0$ and $N^2>0$, and so $h^1(S,\fine)=h^1(S,\fino_S(M))+h^1(S,\fino_S(N))=0$, by Theorem \ref{knutsen2007sharp}. In this case, we have $h^0(S,\fine)=g-d+1$ and $d=M.N=\frac{g-1}{2}$, and so
$$\dim\mathrm{im}\,f_{\fine}=2g-2\cdot\frac{g-1}{2}-3-h^0(M-N)-h^0(N-M)\leq g-2.$$

If $h^2(S,\fino_S(M-N))=0$, then let $i=h^1(S,\fine)$. As in the case where the $\fine$'s are indecomposable with $\ell=0$, we also here get $N^2=0$ if $i>0$, and hence that $N|_C=A$. We thus get one single $g^1_d$ for each curve $C$.



Now suppose $i=0$. This implies that $h^0(S,\fine)= g-d+1$, and so
$$\dim\mathrm{im}\,f_{\fine}\leq 2(g-d-1)-h^0(S,\fino_S(M-N))-1.$$
Using Riemann--Roch together with the assumption that $h^2(S,\fino_S(M-N))=0$, we get $h^0(S,\fino_S(M-N))\geq \frac{1}{2}C^2-2M.N+1=g-2M.N=g-2d$, which gives us
$$\dim\mathrm{im}\,f_{\fine}\leq g-3.$$
\end{proof}

\section{The case where the $\finecax$'s are $\mu_{L}$-stable}\label{stable-section}

In this section, we cover the cases where $\finecax$ is $\mu_{L}$-stable for general $(C,A)$ in $\finw$. It is here not possible to do a parameter count in order to obtain a suitable bound, but we prove here instead that $\dim\ker(\mu_{0,A})\leq 2$, yielding that $\dim W^1_d(C)=d-k$ for $d\leq g-k$ for the curves in question.

Note that, by Assumption (\ref{dominates}), we cannot have $h^0(C,\fino_C(A+K_S|_C))\geq 3$ for general $(C,A)$ in $\finw$, since otherwise, by subtracting points, we would have more than $\dim W^1_d(C)$ dimensions of $g^1_{d-1}$'s, which is impossible. In the following propositions, we therefore only need to consider the cases where $h^0(C,\fino_C(A+K_S|_C))=2$ or $\leq 1$, respectively.

\begin{Prop}\label{H1Eleq1}
Suppose that $\finecax$ is $\mu_L$-stable, and that $h^0(C,\fino_C(A+K_S|_C))=2$, for general $(C,A)$ in $\finw$. Then $\dim \finw\leq g-1+d-k$.
\end{Prop}

\begin{proof}

If $h^0(C,\fino_C(A+K_S|_C))=2$ for the general pairs $(C,A)$ in $\finw$, then by (\ref{base-points}), we can assume that $\fino_C(A+K_S|_C)$ is also base-point free for general $A$, and so these pairs $(C,A+K_S|_C)$ define vector-bundles $\fine_{C,A+K_S|_C}$. If these vector-bundles are non-$\mu_L$-stable for general $(C,A)$ in $\finw$, then by Proposition \ref{non-stable-result}, we get at most $g-1+d-k$ dimensions of pairs $(C,A+K_S|_C)$, and so there must also be at most that many dimensions of pairs $(C,A)$. So suppose the vector-bundles are $\mu_L$-stable.

The vector-bundles $\fine_{C,A+K_S|_C}$ lie inside a sequence
$$0\rightarrow H^0(S,\fino_S(A+K_S|_C))^{\vee}\otimes\fino_S\rightarrow \fine_{C,A+K_S|_C}\rightarrow \fino_C(K_C-A)\rightarrow 0.$$
Now, tensoring this sequence with $\finecaxvee$ and taking global sections, we get
\begin{multline*}
0\rightarrow H^0(S,\finecaxvee)^{\oplus 2}\rightarrow H^0(S,\fine_{C,A+K_S|_C}\otimes\finecaxvee)\rightarrow H^0(C,\fino_C(K_C-A)\otimes\finecaxvee)\\
\rightarrow H^1(S,\finecaxvee)^{\oplus 2}.
\end{multline*}
Since $h^0(S,\finecaxvee)=h^1(S,\finecaxvee)=0$, then by Proposition \ref{coboundary}, $h^0(S,\fine_{C,A+K_S|_C}\otimes\finecaxvee)=\dim\ker\mu_{0,A}$.

Now, suppose first that $\finecax\cong\fine_{C,A+K_S|_C}$. Since we are assuming stability, then it follows that the vector-bundles are simple, and so $h^0(S,\finefine)=1$, and it follows that $\dim\ker\mu_{0,A}= 1$. By (\ref{rho-equality}), $\dim W^1_d(C)= -g+2d-1$, and by putting $d\leq g-k$, we have $\dim W^1_d(C)\leq d-k-1$.

Now assume that $\finecax\ncong\fine_{C,A+K_S|_C}$. Since both $\finecax$ and $\fine_{C,A+K_S|_C}$ are $\mu_L$-stable, then (noting that $\mu_L(\finecax)=\mu_L(\fine_{C,A+K_S|_C})$), we have $h^0(S,\fine_{C,A+K_S|_C}\otimes\finecaxvee)=0$, and so $\dim\ker\mu_{0,A}=0$, and $\dim W^1_d(C)=d-k$ by (\ref{rho-equality}) and (\ref{rho-lg}).
\end{proof}

\begin{Prop}\label{stable-result}
Suppose that $\finecax$ is $\mu_L$-stable, and that $h^0(C,\fino_S(A+K_S|_C))\leq 1$, for general $(C,A)$ in $\finw$. Then $\dim\ker\mu_{0,A}\leq 2$. It follows that if $d\leq g-k$, then $\dim\finw\leq g-1+d-k$.
\end{Prop}

\begin{proof}
Tensoring the sequence (\ref{E-sekvens}) by $\finecax^{\vee}\otimes\fino_S(K_S)$ and taking cohomology, one gets
\begin{multline*}
0\rightarrow H^0(S,(\finecax^{\vee}\otimes\fino_S(K_S)))^{\oplus 2}\rightarrow H^0(S,\finefine\otimes\fino_S(K_S))\rightarrow H^0(C,\finecax^{\vee}\otimes\fino_C(K_C-A))\\
\rightarrow H^1(S,(\finecaxvee\otimes\fino_S(K_S))^{\oplus 2}).
\end{multline*}
We have $H^0(S,\finefine\otimes\fino_S(K_S))=0$ by the stableness assumption, and so it follows that $H^0(C,\finecax^{\vee}\otimes\fino_C(K_C-A))$ injects into $H^1(S,(\finecaxvee\otimes\fino_S(K_S))^{\oplus 2})$.

Since $h^1(S,\finecaxvee\otimes\fino_S(K_S))=h^1(S,\finecax)=h^0(C,\fino_C(A+K_S|_C))$ (by (\ref{h1fine})), which by assumption is $\leq 1$, we have $h^1(C,\finecax^{\vee}\otimes\fino_C(K_C-A))\leq 2$. By Proposition \ref{coboundary}, $\dim\ker\mu_{0,A}\leq 2$.

By (\ref{rho-equality}), it follows that if $\dim \ker\mu_{0,A}\leq 2$ for $[A]$ general in a component $W$ of $W^1_d(C)$, then $\dim W=-g+2d$. By putting $d\leq g-k$, the result follows.
\end{proof}

\begin{Bemerkning}
It is interesting to note that the same result can be obtained by considering the moduli-space $M$ of $\mu_L$-stable vector-bundles of rank $2$ with $c_1=L$ and $c_2=d$ on $S$. It is known (see e.g.\ \cite[Remark, page 768]{kim2006stable}) that the dimension of the tangent space at $\fine$ is given by
$$\dim T_{\fine}M=4c_2-c_1^2-3+h^2(S,\finefineto).$$
Since $h^2(S,\finefineto)=0$ in our case, it follows that the dimension is given by $4d-L^2-3=4d-2g-1$.

By considering all possible injections $\Lambda\hookrightarrow H^0(S,\fine)$, as done in the proof of Proposition \ref{non-stable-result}, we obtain $\dim\finw\leq 2d-1$ using this approach, or equivalenty, $\dim W^1_d(C)\leq 2d-g$. We have $2d-g\leq d-k$ precisely when $d\leq g-k$.
\end{Bemerkning}

\begin{proof}[Proof of Theorem \ref{main}]
Suppose that $\pi:\finw\rightarrow |L|$ dominates. From (\ref{base-points}), we can assume that for general $(C,A)\in \finw$, we have that $|A|$ is base-point free. We can therefore for these $(C,A)$ consider vector-bundles $\finecax$.

If for general $(C,A)$ in $\finw$ we have $\finecax$ non-$\mu_L$-stable, then by Proposition \ref{non-stable-result}, we have $\dim\finw\leq g-1+d-k$. If for general $(C,A)$ in $\finw$ we have $\finecax$ $\mu_L$-stable, then we have the same bound by Propositions \ref{H1Eleq1} and \ref{stable-result}.

Since $\dim|L|=g-1$, the result follows.
\end{proof}

\section{Example of curves on Enriques surfaces with an infinite number of $g^1_{\gon(C)}$'s}\label{example-section}

We here present an example of curves with an infinite number of $g^1_{\gon(C)}$'s.

\begin{Eksempel}\label{non-linear}
Let $S$ be any Enriques surface (which is possibly nodal). Let $L=n(E_1+E_2)$ for $n\geq 3$, where $E_1.E_2=2$, in which case, $k=\mu(L)$ by \cite[Corollary 1.5 (a)]{KL09}. Then there exists a sub-linear system $\mathfrak{d}\subseteq |L|$ of smooth curves such that for general $C\in \mathfrak{d}$, there exist infinitely many $g^1_{\gon(C)}$'s.

Indeed, let $B=\fino_S(E_1+E_2)$, consider the map $f_B:S\rightarrow \tykkp^2$, and let $\mathfrak{d}=f^{\ast}|\fino_{\tykkp^2}(n)|$. This is then a sub-linear system of $|L|$, consisting of all curves that map 4--1 onto curves of $|\fino_{\tykkp^2}(n)|$. By Bertini's theorem, since this linear system is base-point free, the generic elements are smooth.

One constructs infinitely many $g^1_{B.L-4}$'s on a generic smooth curve $C\in\mathfrak{d}$ in the following way: Let $C=f^{-1}(C')$, where $C'$ is smooth in $\fino_{\tykkp^2}(n)$. Then $C$ is also smooth. We let the $g^1_{B.L-4}$'s be $f|_C^{\ast}(\fino_{C'}(1)\otimes\fino_{C'}(-P))$, where $P$ is any point on $C'$. (On $C$, this is the same as subtracting one point $Q$ on $B|_C$ and noting that $|B|_C-Q|$ has three base-points $f^{-1}(f(Q))-Q$ that can also be subtracted.)

By \cite[Corollary 1.6]{KL09}, the minimal gonality is always at most $2$ less than the generic gonality, and the generic gonality is given by $B.L-2$ by \cite[Corollary 1.5]{KL09}, so in our case, it follows that $\gon(C)=B.L-4$. Since $n\geq 3$, we are ensured that the $g^1_{B.L-4}$'s are distinct.

\bigskip
These $g^1_{B.L-4}$'s are, as far as we know, new examples of curves $C$ with infinitely many $g^1_{\gon(C)}$'s. The curves are furthermore non-exceptional.

These curves $C$ are 4--1 coverings of plane curves, and the $g^1_{\gon(C)}$'s are induced from the $g^1_{\gon(C')}$'s. According to the Castelnuovo--Severi inequality (see e.g.\ \cite{Kani}), whenever we have an $m$--1 covering from a curve $C$ to a curve $C'$, if $g(C)>mg(C')+(m-1)(d-1)$, then any base-point free $g^1_d$ on $C$ is induced by a base-point free linear system on $C'$. In particular, if $d=\gon(C)$ and $C'$ has infitely many $g^1_{\gon(C')}$'s, then $C$ also has infinitely many $g^1_{\gon(C)}$'s. However, in this example, $g(C)\leq mg(C')+(m-1)(d-1)$.

Furthermore, by \cite[Corollary 2.3.1]{C-M}, any exceptional curve $C$ has infinitely many $g^1_{\gon(C)}$'s. However, by \cite[Theorem 1.1]{KL13}, the only exceptional curves $C$ on Enriques surfaces are isomorphic to smooth plane quintics and satisfy $C^2=10$. It follows that the curves in our example are non-exceptional.
\end{Eksempel}

\bibliographystyle{amsalpha}
\bibliography{NHR}

\bigskip

\author{Nils Henry Rasmussen and Shengtian Zhou}

\email{nils.h.rasmussen@hit.no, shengtian.zhou@hit.no}

\address{Telemark University College, Dept. of Teacher Education}

\address{L\ae rerskolevegen 40, 3679 NOTODDEN, Norway}
\end{document}